\documentclass[reqno]{amsart}
\usepackage{amssymb}
\usepackage{graphicx}
\usepackage[usenames, dvipsnames]{color}
\usepackage{verbatim}
\usepackage{mathrsfs}
\usepackage{bm}
\usepackage{cite}
\usepackage{color}

\numberwithin{equation}{section}

\newtheorem{theorem}{Theorem}[section]
\newtheorem{corollary}[theorem]{Corollary}
\newtheorem{lemma}[theorem]{Lemma}
\newtheorem{prop}[theorem]{Proposition}

\theoremstyle{definition}
\newtheorem{remark}[theorem]{Remark}

\theoremstyle{definition}
\newtheorem{definition}[theorem]{Definition}

\theoremstyle{definition}

\makeatletter
\def\dashint{\operatorname%
{\,\,\text{\bf-}\kern-.98em\DOTSI\intop\ilimits@\!\!}}
\makeatother

\def\\det{\text{det}}

\newcommand{\ddp}[2]{\frac{\partial#1}{\partial#2}}

\newcommand{\D}{\partial D}
\renewcommand{\S}{\mathcal{S}}
\newcommand{\K}{\mathcal{K}}
\newcommand{\outside}{\mathbb{R}^3\setminus \overline{D}}

\renewcommand*{\Re}{\operatorname{Re}}

\newcommand{\A}{\mathcal{A}}
\newcommand{\de}{\: \mathrm{d}}
\newcommand{\R}{\mathbb{R}}

\def\.5{\frac{1}{2}}

\def\Re{\text{Re}\,}

\newcommand{\RN}[1]{%
  \textup{\uppercase\expandafter{\romannumeral#1}}%
}

\renewcommand{\epsilon}{\varepsilon}

\newcounter{marnote}

\begin{document}


\title[Convex acoustic subwavelength resonators]{The interaction between two close-to-touching convex acoustic subwavelength resonators}

\author[H.G. Li, Y. Zhao]{Haigang Li, Yan Zhao}
\address[H.G. Li]{School of Mathematical Sciences, Beijing Normal University, Laboratory of Mathematics and Complex Systems, Ministry of Education, Beijing 100875, China. }
\email{hgli@bnu.edu.cn}
\address[Y. Zhao]{School of Mathematical Sciences, Beijing Normal University, Laboratory of Mathematics and Complex Systems, Ministry of Education, Beijing 100875, China. }
\email{zhaoyan1999@mail.bnu.edu.cn}

\thanks{H.G. Li was partially supported by  NSFC (11971061)}

\date{\today} 

\maketitle


\begin{abstract}
The Minneart resonance is a low frequency resonance in which the wavelength is much larger than the size of the resonators. It is interesting to study the interaction between two adjacent bubbles when they are brought close together. Because the bubbles are usually compressible, in this paper we mainly investigate resonant modes of two general convex resonators with arbitrary shapes to extend the results of Ammari, Davies, Yu in \cite{ADY}, where a pair of spherical resonators are considered by using bispherical coordinates. We combine the layer potential method for Helmholtz equation in \cite{ADY,AFG2017} and the elliptic theory for gradient estimates in \cite{L,LLY} to calculate the capacitance coefficients for the coupled $C^{2,\alpha}$ resonators, then show the leading-order asymptotic behaviors of two different resonant modes and reveal the dependance of the resonant frequencies on their geometric properties, such as convexity, volumes and curvatures. By the way, the blow-up rates of gradient of the scattered pressure are also presented. 
\end{abstract}

\section{Introduction, Formulation and Main Results}

\subsection{Background}

In recent years, the acoustic properties and applications of bubbly media were investigated in various aspects, since it was discovered that a very small volume fraction of air bubbles in water is enough to modify the effective velocity of sound in the medium, see Minnaert's well-known work \cite{minnaert1933musical} and \cite{ ammari2018minnaert, devaud2008minnaert}. The enhancement of their acoustic signature allows ultrasonic techniques to detect, localize, and characterize them inside a visco-elastic opaque medium \cite{CTL2012}. The extraordinary acoustic properties of bubbly media are also used to design new acoustic materials, because the Minnaert resonance of the bubbles persists not only in liquid but also in the soft elastic medium, where the resonators are commonly constructed from a material that has greatly different material parameters to the background medium.  There are many interesting works in physics as well on the acoustic bubble problem, see  \cite{LPLFLT2015} and the references therein. In practice, structures made from subwavelength resonators (a type of metamaterial) also have been used for a wide variety of wave-guiding applications, e.g. \cite{AFLZ2019, kushwaha1998sound, AFG2017, ad, ADOY2019}.

The Minnaert resonance is a low frequency resonance in which the wavelength is much larger than the size of the bubbles \cite{devaud2008minnaert}. It is very interesting to study the interaction between two adjacent bubbles when they are brought close together. Using $\epsilon$ to express the distance between them, Ammari, Davies, Yu \cite{ADY} recently studied the behavior of the coupled subwavelength resonant modes of a pair of spherical resonators. They employed the layer potential method to represent solutions and proved that the leading-order behavior of the resonant modes is determined by the so-called capacitance coefficients \cite{AFLZ2019}, which are well known in the setting of electrostatics, and then calculated them by using bispherical coordinates. However, considering that the bubbles are usually compressible and not necessarily spherical, and in view of shape optimization for applications, it is necessary and important to study the resonators with arbitrary shapes. Motivated by \cite{ADY}, in this paper we mainly deal with the case of general convex resonators in order to explore more applications of subwavelength resonant media.

This multi-scale problem involves the high contrast parameter, $\delta$, of the material density, the small separation distance, $\epsilon$, between the two resonators, the volume and local curvatures of each resonator. From the technical point of view, the technique in this paper breaks the restriction to spheres due to the use of bi-spherical coordinates on the circular case, which allows us to distinguish the effect from the volume (global information) and curvatures (local information) on the resonant frequencies. On the other hand, it can be applied to deal with ellipsoidal resonators, $m$-convex resonators, and even more general convex resonators of arbitrary shape. The main result of this paper shows that both two sub-wavelength resonant frequencies $\omega_i(\delta,\epsilon)\to 0$ as $\delta\to0$ with $\epsilon$ chosen to be a suitable function of $\delta$, and their asymptotic behaviors are of different scales when the distance $\epsilon$ between the resonators is sufficiently small (see Theorem \ref{the:main} and Remark \ref{remark1} for more details).

We first use the layer potential method to perform an asymptotic analysis in terms of the material contrast \cite{ammari2018minnaert, AFLZ2019}, examine the behavior of the eigenmodes when the resonators are brought together \cite{ADY}, and reduce the leading order behavior of the resonant modes to the calculation of the so-called capacitance coefficients. Then, we adapt the local gradient estimates for the Laplace equation developed in \cite{L,llby} to compute asymptotic expansions of the capacitance coefficients for the $C^{2,\alpha}$ resonators with $m$-convex boundaries, $m\geq2$, in the exterior domain. Compared with \cite{L,LLY}, the difference here is that the domain we consider in this paper is unbounded while it is bounded in \cite{L,LLY}. Especially, we need to control the asymptotic behavior at infinity and extend the results in bounded domains to unbounded domains. Finally, similarly as in \cite{ADY}, we show that there are two subwavelength resonant frequencies (see Lemma \ref{lem:two_modes} below) having different asymptotic behaviors
and depending on the contrast of the density, the convexity, volumes, curvatures of the resonators, which can be neatly expressed when the separation distance is chosen as a function of the material contrast. This may have significant implications for the design of acoustic meta-materials with multi-frequency or broadband functionality. As a consequence, we show the dependence of the blow-up rate of the gradient on the geometry of the resonators, is similar to electrostatics \cite{akl, gorb2015singular, bly, lekner2011near,y1}  and linear elasticity \cite{KY, bll, bll2, Li} and the references therein, provided the parameters of the inclusions degenerate to infinity. Here we remark that our method is different from the method used in \cite{ADY} where the authors use bi-spherical coordinates to derive explicit representations for the capacitance coefficients and capture the system's resonant behavior at leading order. Finally, we would like to point out that in this present paper and \cite{ADY} the material parameters are positive, which are different from the surface plasmons studied in \cite{yu2018plasmonic, bonnetier2013spectrum, khattak2019linking}, where the electromagnetic inclusions have negative permittivities.

\subsection{Helmholtz Equations}
In this paper, we study the scattering problem in dimension three of the Helmholtz equation in a high-contrast composite structure, especially when the inclusions are closely spaced. Assume that in a homogeneous background medium, with density $\rho$ and bulk modulus $\kappa$, there is a pair of convex resonators $D_1$ and $D_2$ $\epsilon$-apart, for a small constant $\epsilon>0$. Let $\rho_b$ and $\kappa_b$ denote the density and bulk modulus of the medium that occupies $D_1\cup D_2$. In this paper we assume that $D_1$ and $D_2$ are of arbitrary shape, with $C^{2,\alpha}$ boundaries, $0<\alpha<1$. We follow the notations in \cite{ADY} and introduce the auxiliary parameters
\begin{equation}\label{def_vb}
v=\sqrt{\frac{\kappa}{\rho}},~ v_b=\sqrt{\frac{\kappa_b}{\rho_b}};\quad k= \frac{\omega}{v}, ~k_b= \frac{\omega}{v_b}, 
\end{equation}
which are the wave's speeds and wave's numbers in $\R^3\backslash \overline{D_1\cup D_2}$ and in $D_1\cup D_2$, respectively, and two dimensionless contrast parameters
\begin{equation}\nonumber
\delta=\frac{\rho_b}{\rho},\quad\tau=\frac{k_b}{k}=\frac{v}{v_b}=\sqrt{\frac{\rho_b\,\kappa}{\rho\,\kappa_b}}. 
\end{equation}
Then the acoustic pressure $u$ produced by the scattering of an incoming plane wave $u^{in}$ satisfies
\begin{equation}\label{equ:hel}
\left\{
\begin{aligned}
(\Delta+k^2)\,u&=0,  \quad \mbox{in}~  \R^3 \backslash \overline{D_1\cup D_2},\\
 (\Delta+k_b^2)\,u&=0,  \quad \mbox{in}~  D_1\cup D_2,\\
u|_+-u|_-&=0,  \quad \mbox{on}~  \partial D_1\cup\partial D_2,\\
 \delta \frac{\partial u}{\partial \nu}\Big|_{+}-\frac{\partial u}{\partial \nu}\Big|_{-}&=0,  \quad \mbox{on}~  \partial D_1\cup \partial D_2,\\
\end{aligned} 
\right.
\end{equation}
along with the Sommerfeld radiation condition:
$$(\frac{\partial}{\partial|x|}-ik)(u-u^{in})(x)=O(|x|^{-2}),\quad\mbox{as} ~|x|\to\infty, $$
where the subscripts $+$ and $-$ denote the evaluation from outside and inside $\partial D_{i}$, respectively. We assume that $v$, $v_b$, $\tau$ are all $O(1)$, while the contrast of the density $0< \delta \ll 1$. A classic example satisfying these assumptions is a collection of air bubbles in water, often known as Minnaert bubbles, in which case $\delta \approx 10^{-3}$.

\subsection{Our Domains}

Before stating the main results we first describe our domains precisely. We use $x=(x',x_{3})$ to denote a point in $\mathbb{R}^{3}$, where $x'=(x_{1},x_{2})$. Let $D_{1}^{0}$ and $D_{2}^{0}$ be a pair of touching convex subdomains in $\R^3$, such that
\begin{equation}\label{D10}
D_{1}^{0}\subset\{(x',x_{3})\in\mathbb R^{3}| x_{3}>0\},\quad D_{2}^{0}\subset\{(x',x_{3})\in\mathbb R^{d}| x_{3}<0\},
\end{equation}
with $\{x_{3}=0\}$ being their common tangent plane, and $\partial D_{1}^{0}\cap\partial D_{2}^{0}=\{(0',0)\}$. 
We further assume that the $C^{2,\alpha}$ norms of
$\partial{D}_{i}$ $(i=1,2)$ are bounded by some constants to guarantee that their volumes  are independent of $\varepsilon$ (or of a scale of different order compared to $\varepsilon$). Translating $D_{1}^{0}$ by $\varepsilon$ along the $x_{3}$-axis,  we have $D_{1}^{\varepsilon}:=D_{1}^{0}+(0',\varepsilon)$. When there is no confusion, we always drop the superscripts, and denote
$$D_{1}:=D_{1}^{\varepsilon},\quad D_{2}:=D_{2}^{0},\quad  D:={D_{1}\cup{D}_{2}},$$
and
$$\Omega=\R^3\backslash\overline D,\quad\quad \Omega^0=\R^3\backslash\overline{{(D_1^0\cup D_2^0)}}.$$
We remark that here our domains $\Omega$ and $\Omega^0$ are both unbounded, different from the case of bounded domain considered previously in \cite{L,LLY}.

Near the origin, similarly as in \cite{bll,L}, let us set the two boundary points
\begin{equation*}\label{P1P2}
P_{1}=\left(0',\varepsilon\right)\in \partial D_1,\quad\,P_{2}=\left(0',0\right)\in \partial D_2.
\end{equation*}
By our smoothness assumptions, there exists a small universal constant $R_{0}<1$ such that the portions of
 $\partial{D}_{i}$  near $P_{i}$  can be parameterized by $(x',\varepsilon+h_{1}(x'))$ and $(x',h_{2}(x'))$,
respectively, namely,
$$x_{3}=\varepsilon+h_{1}(x'),~\mbox{and}~\,x_{3}=h_{2}(x'), \quad\mbox{for}~~ |x'|<2R_{0} .$$
By the convexity assumptions on $\partial{D}_{i}$,  $h_{1}$ and $h_{2}$  satisfy
\begin{equation}\label{h1h200}
\varepsilon+h_{1}(x')>h_{2}(x'),\quad\mbox{for}~~|x'|<2R_{0},
\end{equation}
\begin{equation}\label{h1h20}
h_{1}(0')=h_{2}(0')=0,\quad\nabla_{x'}h_{1}(0')=\nabla_{x'}h_{2}(0')=0,
\end{equation}
and
\begin{equation}\label{h1h2}
\|h_{1}\|_{C^{2,\alpha}(B'_{2R_{0}})}+\|h_{2}\|_{C^{2,\alpha}(B'_{2R_{0}})}\leq{C}.
\end{equation}
For simplicity, we assume that 
\begin{equation}\label{equ:D1D2}
\begin{aligned}
&(h_1-h_2)(x')=\Lambda|x'|^m+O(|x'|^{m+1}),\quad\mbox{for}~~ |x'|<2R_0, \\
 &|\nabla_{x'}h_1(x')|,|\nabla_{x'}h_2(x')|\leq C|x'|^{m-1}, \quad\mbox{for}~~ |x'|<2R_0,
\end{aligned}
\end{equation}
and $m$ is an integer, $m\geq2$. Here and throughout the paper, we use the notation $O(A)$ to denote a quantity that can be bounded by $CA$, where $C$ is some positive constant independent of $\varepsilon$ and $\delta$. We call such inclusions $m$-convex inclusions. For example, there is a curvilinear square with rounded-off angles, with boundary given by
$$|x_{1}|^{m}+|x_{2}|^{m}+|x_{3}|^{m}=a^{m},\quad\mbox{if}~\,m>2,$$
for some positive $a$, representing a half width of the inclusion, see Figure \ref{fig1} and Figure \ref{fig2}. We would like to point out that in dimension two this type of inclusion is very close to the Vigdergauz inclusion (see Fig. 3 in \cite{LL}), which has also a nearly square shape and is proved to minimize the elastic energy under the same volume fraction of hard inclusions [19]. More discussions on the connection between $m$-convex inclusion and the Vigdergauz inclusion can refer to page 7-8 in \cite{LL}. So it is interesting to consider $m$-convex inclusion and use the simple curve boundary (1.7) to describe narrow region between two inclusions with nearly square shapes. This is also one motivation of this paper.
\begin{figure}
\begin{minipage}{0.450\linewidth}
\centerline{\includegraphics[width=6.0cm]{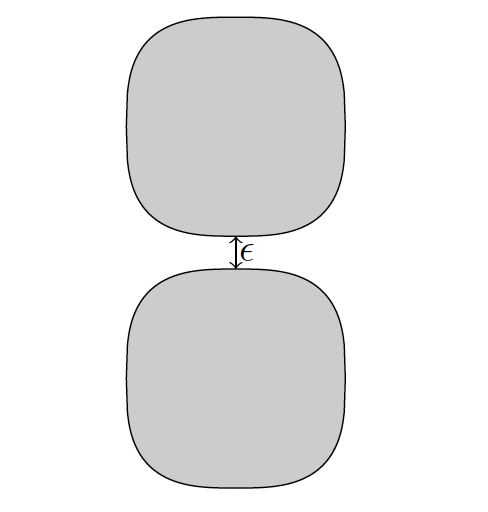}}
\caption{$m=3$.}
\label{fig1}
\end{minipage}
\begin{minipage}{0.450\linewidth}
\centerline{\includegraphics[width=6.0cm]{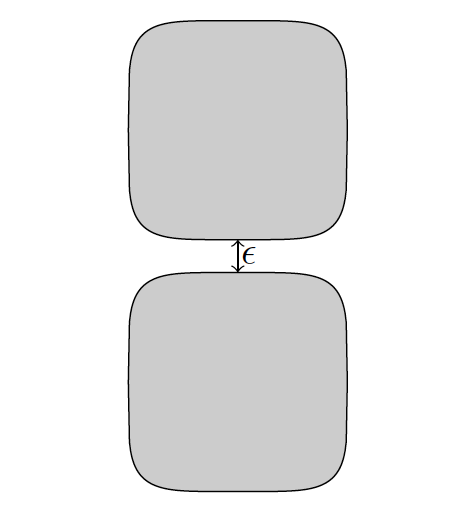}}
\caption{$m=5$.}
\label{fig2}
\end{minipage}       
\end{figure}
For more general domains, see Remark \ref{rem1.8} below. Here and throughout this paper, unless otherwise stated, $C$ denotes a constant, whose values may vary from line to line, depending only on $n$,  $\|\partial{D}_{1}\|_{C^{2,\alpha}}$ and $\|\partial{D}_{2}\|_{C^{2,\alpha}}$, not on $\varepsilon$. We call a constant having such dependence a
{\it universal constant}.

\subsection{Layer Potential for the Helmholtz Equations} In this subsection, we follow the definitions and the notations used in \cite{ADY}. More properties on the layer potential theory and the Neumann--Poincar\'e operator can be found in \cite{AFGZbook2018}. For the Helmholtz equation
$$\Delta u+k^2 u=0,$$
let $G^k$  be the (outgoing) Helmholtz Green's function
\begin{equation*}
G^k(x,y) := -\frac{e^{\mathrm{i} k|x-y|}}{4\pi|x-y|}, \quad x,y \in \R^3,~ k\geq 0,
\end{equation*}
and $\S_{D}^k: L^2(\partial D) \rightarrow H_{\textrm{loc}}^1(\R^3)$ be the corresponding single layer potential, see \cite{12}, defined by
\begin{equation*}
\S_D^k[\phi](x) := \int_{\D} G^k(x,y)\phi(y) \de \sigma(y), \quad x \in \R^3, \, \phi\in L^2(\partial D).
\end{equation*}
Here $H_{\textrm{loc}}^1(\R^3)$  is the usual Sobolev space. We also define the Neumann--Poincar\'e operator $\K_D^{k,*}: L^2(\partial D) \rightarrow L^2(\D)$ by
\begin{equation*}
\K_D^{k,*}[\phi](x) := \int_{\partial D} \frac{\partial }{\partial \nu_x}G^k(x,y) \phi(y) \de \sigma(y), \quad x \in \partial D,
\end{equation*}
where $\partial/\partial \nu_x$ denotes the outward normal derivative at $x\in\D$.
By recalling the transmission conditions for the single layer potential on $\D$ \cite{AFGZbook2018}, in particular, for any $\varphi\in L^2(\D)$,
\begin{equation}\label{eq:jump}
\ddp{}{\nu}\S_D[\varphi]|_\pm=\Big(\pm\frac{1}{2} I+\K_D^*\Big)[\varphi],\quad\mbox{on}~\partial{D}.
\end{equation}
Using the exponential power series we can derive an expansion for $\S_D^k$, given by
\begin{equation} \label{eq:S_series}
\S_D^k= \S_{D} + \sum_{n=1}^{\infty} k^n \S_{D,n},
\end{equation}
where, for $n=0,1,2,\dots$,
\begin{equation*}
\S_{D,n}[\phi](x):=-\frac{\mathrm{i}^n}{4\pi n!} \int_{\D} |x-y|^{n-1}\phi(y) \de\sigma(y), \quad x \in \R^3, \, \phi\in L^2(\partial D),
\end{equation*}
and $\S_D:=\S_{D,0}$ is the Laplace single layer potential. It is well known that $\S_D: L^2(\D) \rightarrow H^1(\D)$ is invertible \cite{AFGZbook2018}. Similarly, for $\K_D^{k,*}$ we have that 
\begin{equation} \label{eq:K_series}
\K_D^{k,*}=\K_D^*+\sum_{n=1}^{\infty} k^n\K_{D,n},
\end{equation}
where, for $n=0,1,2,\dots$,
\begin{equation*}
\K_{D,n}[\phi](x):=-\frac{\mathrm{i}^n(n-1)}{4\pi n!} \int_{\D} |x-y|^{n-3}(x-y)\cdot \nu_x \,\phi(y) \de\sigma(y), \quad x \in \R^3, \, \phi\in L^2(\partial D),
\end{equation*}
and $\K_D^*:=\K_{D,0}$ is the Neumann--Poincar\'e operator corresponding to the Laplace equation. In fact, by the following Lemma, as $k, \epsilon\to0$ we could write $\S_D^k= \S_{D}+O(k)$ and $\K_D^{k,*}=\K_D^*+O(k)$ in the relevant operator norms respectively .

\begin{lemma}(\cite{AFGZbook2018})\label{lemma11}
The norms $||S_{D,n}||_{B(L^2(\partial D),H^1(\partial D))}$ and $||\K_{D,n}||_{B(L^2(\partial D),L^2(\partial D))}$ are uniformly bounded for $n\geq 1$ and $0<\epsilon\ll1$. Moreover, the series \eqref{eq:S_series} and \eqref{eq:K_series} are uniformly convergent for $\epsilon>0$, in $B(L^2(\partial D), H^1(\partial D))$ and $B(L^2(\partial D), L^2(\partial D))$, respectively.
\end{lemma}

The solutions to \eqref{equ:hel} can be represented as follows: 
\begin{equation} \label{eq:layer_potential_representation}
u = \begin{cases}
u^{in}(x)+\S_{D}^k[\psi](x), & x\in\outside,\\
\S_{D}^{k_b}[\phi](x), & x\in D=D_{1}\cup D_{2},
\end{cases}
\end{equation} 
for some surface potentials $(\phi,\psi)\in L^2(\D)\times L^2(\D)$, which must be chosen so that $u$ satisfies the two transmission conditions across $\D$. Using the jump relation \eqref{eq:jump}, this is equivalent to $(\phi,\psi)$ satisfying 
\begin{equation} \label{eq:A_matrix_equation}
\A(\omega,\delta)\begin{pmatrix} \phi \\ \psi \end{pmatrix} = \begin{pmatrix} u^{in} \\ \delta\ddp{u^{in}}{\nu_x}\end{pmatrix},
\end{equation}
where
\begin{equation} \label{matrix_A}
\A(\omega,\delta):=
\begin{pmatrix}
	\S_D^{k_b} & -\S_D^k \\
	-\frac{1}{2}\mathcal{I}+\K_D^{k_b,*} & -\delta(\frac{1}{2}\mathcal{I}+\K_D^{k,*}) 
\end{pmatrix},
\end{equation}
and $\mathcal{I}$ is the identity operator on $L^2(\D)$, see e.g. \cite{AFGZbook2018} for more details.

\subsection{Resonant Frequencies and the Capacitance Coefficients}

In light of the representation \eqref{eq:layer_potential_representation}, we can define the notion of resonance to be the existence of a non-trivial solution when the incoming field $u^{in}$ is zero.

\begin{definition} \label{defn:resonance}
	For a fixed $\delta$, we define a resonant frequency (or eigenfrequency) to be $\omega\in\mathbb{C}$ such that there exists a non-trivial solution to
	\begin{equation*}\label{eq:res}
	\A(\omega,\delta)
	\begin{pmatrix}
	\phi \\ \psi
	\end{pmatrix}
	=
	\begin{pmatrix}
	0 \\ 0
	\end{pmatrix},
	\end{equation*}
	where $\A(\omega,\delta)$ is defined in \eqref{matrix_A}.
	For each resonant frequency $\omega$, we define the corresponding resonant mode (or eigenmode) as
	\begin{equation} \label{eq:eigenmode_representation}
	u = \begin{cases}
	\S_{D}^{k_b}[\phi](x), & x\in D=D_{1}\cup D_{2}, \\
	\S_{D}^k[\psi](x), & x\in\outside.
	\end{cases}
	\end{equation}
\end{definition}

\begin{definition}\label{def12}
	We define a subwavelength resonant frequency to be a resonant frequency $\omega=\omega(\delta)$ such that $\omega(0)=0$, and $\omega$ depends on $\delta$ continuously.
\end{definition}
The resonant modes \eqref{eq:eigenmode_representation} are determined only up to normalization. We will choose the normalization to be such that $u_n\sim1$ on $\D$ for all small $\delta$ and $\epsilon$, see Theorem \ref{thm:gradient_bounds} below. When $\omega=0$ and $\delta=0$,
\begin{equation} \label{matrix_defn}
\A(0,0):=
\begin{pmatrix}
	\S_D & -\S_D \\
	-\frac{1}{2}\mathcal{I}+\K_D^{*} & 0
\end{pmatrix},
\end{equation}
Since $\S_D$ is invertible, it follows that $\dim ker\A(0,0)=\dim ker(-\frac{1}{2}\mathcal{I}+\K_D^{*})$. By the argument in Lemma 2.12 of \cite{ad}, we know that
$$ker(-\frac{1}{2}\mathcal{I}+\K_D^{*})=\text{span} \Big\{\S_D^{-1}[\chi_{\partial{D}_{1}}],\S_D^{-1}[\chi_{\partial{D}_{2}}]\Big\}.$$

Notice that each resonant mode in fact has two resonant frequencies associated to it with real parts that differ in sign.  Here, we use $\omega_n$ to denote the resonant frequency associated to $u_n$ that has a positive real part. By virtue of Lemma \ref{lem:26} below, $-\omega_n$ is also a resonant frequency associated to the mode $u_n$. See Lemma \ref{lem:two_modes} and Remark \ref{remark15} below.

By the theory of Gohberg and Sigal \cite{GL}, the following Lemma is obtained in \cite{ADY}.
\begin{lemma} \label{lem:two_modes}
	(\cite{ADY}) There exist two subwavelength resonant modes, $u_1$ and $u_2$, with associated resonant frequencies $\omega_1$ and $\omega_2$ with positive real part, labelled such that $\Re(\omega_1)<\Re(\omega_2)$.
\end{lemma}

Now let $\psi_1,\psi_2\in L^2(\D)$ be given by
\begin{equation} \label{eq:psi_defs}
\S_D[\psi_1]=\begin{cases}
1&\text{on } \D_1,\\
0&\text{on } \D_2,
\end{cases}
\qquad\mbox{and}~
\S_D[\psi_2]=\begin{cases}
0&\text{on } \D_1,\\
1&\text{on } \D_2.
\end{cases}
\end{equation}
It follows that
\begin{equation} \label{kernel_basis}
ker\left(-\frac{1}{2}\mathcal{I}+\K_D^*\right)=\text{span}\{\psi_1,\psi_2\}.
\end{equation}
Let $v_j=S_D[\psi_j]$ be defined as the extension of \eqref{eq:psi_defs} to $\mathbb{R}^{3}\backslash\overline{D}$. Then $v_j$ is the unique solution to \begin{equation} \label{equ:vj}
\begin{cases}
\Delta v_j = 0, & \text{in } ~\R^3 \backslash \overline{D_1\cup D_2}, \\
v_j=\delta_{ij}, & \text{on } \D_i,\\
v_j(x)=O\left(\tfrac{1}{|x|}\right), & \text{as } |x|\to\infty.
\end{cases}
\end{equation}
We define the capacitance matrix $\mathcal{C}=(C_{ij})_{2\times2}$ as follows:
\begin{equation}\label{co1}
C_{ij}:=-\int_{\D_i} \psi_j \de \sigma, \quad i,j=1,2.
\end{equation}
Thanks to the jump relation \eqref{eq:jump} and the fact that $\psi_j$ satisfies the \eqref{kernel_basis}, 
$$\int_{\D_i} \ddp{S_D[\psi_j]}{\nu}\bigg|_+ \de \sigma=\int_{\D_i} \Big(\frac{1}{2} I+\K_D^*\Big)[\psi_j] \de \sigma=\int_{\D_i} \psi_j \de \sigma.$$
Thus, we can write the capacitance coefficients in the form
\begin{equation} \label{eq:cap_reform}
C_{ij}=-\int_{\D_i} \ddp{v_j}{\nu}\bigg|_+ \de \sigma, \quad i,j=1,2.
\end{equation}

For $C_{ii}$, by applying the divergence theorem, and in view of Proposition 2.74 and 2.75 in \cite{GB}, which imply that $\frac{\partial v_j}{\partial r}\sim O(\frac{1}{|x|^2})$ as $|x|\to\infty$, we have 
\begin{equation}\label{co2}
C_{ii}=-\int_{\partial D_i}\frac{\partial v_i}{\partial \nu}\Big|_{+}d\sigma=\int_{\R^3\backslash \overline{D}}\left|\nabla v_i\right|^2.
\end{equation}

Denote
\begin{equation}\label{Emn}
\rho_{m}(\epsilon)=
\begin{cases}
|\log\epsilon|,&m=2,\\
\frac{1}{\epsilon^{1-2/m}},&m>2;
\end{cases}
\quad\mbox{and}~
E_{m}(\epsilon) = \begin{cases}
\epsilon^{1/4}{|\log\epsilon|} , & m=2,\\
\epsilon^{\frac{1}{2m}}, & m>2;
\end{cases}
\end{equation}

and
\begin{equation}\label{gamma_m}
\mathcal{L}_{m}:=2\pi\int_{0}^{+\infty}\frac{r\,dr}{1+r^{m}}, m>2 ;\quad \mathcal{L}_{2}:=\pi\end{equation} 
Then, we adapt the method in \cite{L} for gradient estimates for the Dirichlet problem in bounded domain to exterior domains $\R^3\backslash \overline{D}$ and obtain the following asymptotic formula for the capacitance coefficients.

\begin{prop}\label{the:1}
Suppose that $D_{1}$ and $D_{2}$ satisfy \eqref{h1h200}-\eqref{equ:D1D2}. Let $C_{ij}$ be defined by \eqref{eq:cap_reform}. Then  

(i) there exist constants $M_i$, depending only on $D_1^0$ and $D_2^0$, such that
\begin{equation}\label{equ:m>2}
C_{ii}= \frac{\mathcal{L}_{m}}{\Lambda^{2/m}}\rho_{m}(\epsilon)+M_i+O(E_{m}(\epsilon)),\quad\,i=1,2,
\end{equation} 
where  $\mathcal{L}_{m}$, $\rho_{m}(\epsilon)$, $E_{m}(\epsilon)$ are defined in \eqref{Emn} and \eqref{gamma_m}.

(ii) there exists a constant $C>0$, independent of $\epsilon$, such that
\begin{equation}\label{equ:C1112}
\frac{1}{C}<C_{11}+C_{12}<C,\quad\frac{1}{C}<C_{21}+C_{22}<C.
\end{equation}  
\end{prop} 
Denote the rescaled form of $C_{ij}$ by
\begin{equation} \label{equ:cij}
\overline{C}_{ij}=\frac{1}{|D_i|}C_{ij}, \quad i,j=1,2,
\end{equation}
where $|D_i|$ denotes the volume of $D_i$. Let $\lambda_n$ denote the eigenvalues of matrix $\overline{\mathcal{C}}=(\overline{C}_{ij})_{2\times2}$. A simple calculation yields
\begin{equation}\label{equ:eig}
\lambda_n=\frac{1}{2}\left(\overline{C}_{11}+\overline{C}_{22}+(-1)^n\sqrt{(\overline{C}_{11}-\overline{C}_{22})^2+4\overline{C}_{12}\overline{C}_{21}}\right),\quad n=1,2.
\end{equation}
Then the resonant frequencies are
determined by $\lambda_1$ and $\lambda_2$.  The following relationship is from Lemma 4.1 in \cite{ADY}, which actually holds for bounded domains with Lipschitz boundaries, see Lemma 2.1 of \cite{AFLZ2019}. For the reader's convenience, we give some details of the proof.

\begin{lemma} (\cite{ADY})\label{lem:26}
As the contrast $\delta\to0$, the subwavelength resonant frequencies of two resonators $D_1$ and $D_2$ are given by
\begin{equation*}
\omega_n=\sqrt{\delta v_b^2 \lambda_n}+O(\delta),\quad n=1,2,
\end{equation*}
where $v_b=\sqrt{\frac{\kappa_b}{\rho_b}}$ defined in \eqref{def_vb}, and $\lambda_1$, $\lambda_2$ are defined in \eqref{equ:eig}.
\end{lemma}
\begin{proof}
The proof is the same as that of Lemma 4.1 of  \cite{ADY}. Suppose that $(\phi,\psi)$ is a solution to \eqref{eq:res} for small $\omega=\omega(\delta)$. From \eqref{eq:res} and the asymptotic expansions \eqref{eq:S_series} and \eqref{eq:K_series}, we have  
\begin{equation} \label{eq:A_matrix}
\begin{pmatrix}\S_D+k_b\S_{D,1} & -\S_D-k\S_{D,1} \\ -\frac{1}{2}I+\K_D^*+k_b^2\K_{D,2} & -\delta(\frac{1}{2}I+\K_D^*) \end{pmatrix}\begin{pmatrix} \phi \\ \psi \end{pmatrix} = \begin{pmatrix}  O(\omega^2) \\ O(\delta\omega+\omega^3)\end{pmatrix}.
\end{equation}
Therefore, 
\begin{equation}\label{eq:first}
\S_D[\phi-\psi]+k_b\S_{D,1}[\phi]-k\S_{D,1}[\psi]=O(\omega^2).
\end{equation}
Since $\S_D: L^2(\D) \rightarrow H^1(\D)$ is invertible \cite{AFGZbook2018}, and by virtue of Lemma \ref{lemma11}, $\|S_{D,1}\|_{B(L^2(\partial D),H^1(\partial D))}$ is bounded, it follows that $\phi=\psi+O(\omega)$ in $L^2(\partial D)$. Now by virtue of Lemma~2.1 of \cite{AFLZ2019} (where the boundary $\partial D=\partial D_1\cup\partial D_2$ is assumed to be Lipschitz), we have, for any $\varphi\in L^2(\D)$
\begin{equation} \label{k_properties}
\begin{split}
\int_{\D_i}\left(-\frac{1}{2}I+\K_D^{*}\right)[\varphi]\de\sigma=0,
\qquad&\int_{\D_i}\left(\frac{1}{2}I+\K_D^{*}\right)[\varphi]\de\sigma=\int_{\D_i}\varphi\de\sigma,\\
\int_{\D_i} \K_{D,2}[\varphi]\de\sigma&=-\int_{D_i}\S_D[\varphi]\de x.
\end{split}
\end{equation}
Using \eqref{eq:A_matrix} again, we get
\begin{equation}\label{eq:second}
\left(-\frac{1}{2}I+\K_D^*+k_b^2\K_{D,2}\right)[\phi]-\delta\left(\frac{1}{2}I+\K_D^*\right)[\psi]=O(\delta\omega+\omega^3) .
\end{equation}
Integrating \eqref{eq:second} over $\D_i$, using \eqref{k_properties} yields
\begin{equation} \label{eq:313}
-k_b^2\int_{D_i}\S_D[\phi]\de\sigma -\delta\int_{\D_i}\psi\de\sigma =O(\delta\omega+\omega^3),\quad\,i=1,2.
\end{equation}

In light of \eqref{kernel_basis}, a nontrivial solution $\phi$ to \eqref{eq:second} can be written as
\begin{equation} \label{eq:psi_basis}
\phi=a_1\psi_1+a_2\psi_2+O(\omega^2+\delta),
\end{equation}
for some constants $a_1$ and $a_2$. Then, combining $\phi=\psi+O(\omega)$, we have $\psi=a_1\psi_1+a_2\psi_2+O(\omega+\delta)$. Substituting this into \eqref{eq:313}, we have 
\begin{equation} \label{equ:relation}
\begin{cases}
 -k_b^2a_1|D_1|-\delta a_1\int_{\partial D_1}\psi_1  \de\sigma-\delta a_2\int_{\partial D_1}\psi_2\de\sigma=O(\omega^3+\delta\omega+\delta^2),\\
-k_b^2a_2|D_2|-\delta a_1\int_{\partial D_2}\psi_1  \de\sigma-\delta a_2\int_{\partial D_2}\psi_2\de\sigma=O(\omega^3+\delta\omega+\delta^2).
\end{cases}
\end{equation}
Thus, by using \eqref{co1} and \eqref{eq:cap_reform}, up to an error of order $O(\omega^3+\delta \omega+\delta^2)$, we reach the eigenvalue problem
\begin{equation}\label{eigenvalue}
\tilde{C}\begin{pmatrix}a_1\\a_2\end{pmatrix}
=\frac{k_b^2}{\delta}\begin{pmatrix}a_1\\a_2\end{pmatrix}.
\end{equation}
This implies $\omega_n=\sqrt{\delta v_b^2\lambda_n}+O(\delta)$.\; 
\end{proof}

\begin{remark}\label{remark15}
By \eqref{eigenvalue}, we know that $\frac{\omega_i^2}{\delta v_b^2}=\frac{k_b^2}{\delta}$ is the eigenvalue $\lambda_i$ of $\tilde{C}$, up to an error of order $O(\omega^3+\delta \omega+\delta^2)$, $i=1,2$, so we have $\omega_i^2= \delta v_b^2\lambda_i+O(\delta^2)$. We use $\omega_i$ to denote the resonant frequency that has a positive real part, which means that we write $\omega
_i$ as $\omega_i=\sqrt{\delta v_b^2\lambda_i}+O(\delta)$.  Here we would like to point out that in fact the resonant frequencies $\omega_1$ and $\omega_2$ are not real-valued,  their imaginary parts will appear in the $O(\delta)$ term. Because we are studying resonators in an unbounded domain, energy is lost to the far field meaning that the resonant frequencies have negative imaginary parts \cite{ad, ammari2018minnaert, AFLZ2019}. Hence, the leading-order terms in the expansions for $\omega_1$ and $\omega_2$ (given in Lemma \ref{lem:26}) are real valued and the imaginary parts will appear in higher-order terms in the expansion. However, by Lemma \ref{lemma11} only the leading-order terms in the asymptotic expansion \eqref{eq:S_series} and \eqref{eq:K_series} have singularities as the resonators are moved close together, we do not consider the higher-order expansions in this work.
\end{remark}

\subsection{Main Results}\label{subsec_MR}

Setting 
\begin{equation}
\sigma_1:=\overline{C}_{11}+\overline{C}_{12}=\frac{1}{|D_{1}|}(C_{11}+C_{12}),\quad \sigma_2:=\overline{C}_{22}+\overline{C}_{21}=\frac{1}{|D_{2}|}(C_{21}+C_{22}),
\end{equation}
by using \eqref{equ:C1112}, we have
\begin{equation}\label{sigma_i}
\frac{1}{C}<\sigma_i<C.
\end{equation}
Furthermore, let
\begin{equation}\label{def_C0}
C_{*}=\frac{\overline{C}_{11}\sigma_2+\overline{C}_{22}\sigma_1}{\overline{C}_{11}+\overline{C}_{22}}.
\end{equation}
Thus, by virtue of \eqref{sigma_i} and \eqref{equ:C1112}, there exists some constant $C>0$, independent of $\epsilon$, such that
\begin{equation}\label{sigma_1}
\frac{1}{C}<C_{*}<C.
\end{equation}

Thanks to the explicit calculation of the capacitance coefficients in Proposition \ref{the:1} and the relation between the subwavelength resonant frequencies and the capacitance proved in the Lemma \ref{lem:26}, we could compute the asymptotic expansions of the subwavelength resonant frequencies  in terms of $\epsilon$ and $\delta$ as $\delta\to0$. Now we state our main result as follows. It is valid for the convex resonators of arbitrary shape.

\begin{theorem}\label{the:main}
Let $D_{1}$ and $D_{2}$ be defined as above and satisfy  \eqref{h1h200}-\eqref{equ:D1D2}. Then as the contrast of the density $\delta=\frac{\rho_b}{\rho}\to0$, the resonant frequencies of $D_1$ and $D_2$, with sufficiently small separation distance $\epsilon$, are given by
\begin{equation}\label{equ:lambda_m}
\begin{aligned}
\omega_1&=\sqrt{\delta v_b^2 C_{*}}+O\left(\sqrt{\frac{\delta}{\rho_m(\epsilon)}}+\delta\right),  \\
\omega_2&=\sqrt{\delta v_b^2\Big(\frac{1}{|D_1|}+\frac{1}{|D_2|}\Big)\frac{\mathcal{L}_{m}}{\Lambda^{2/m}}\rho_{m}(\epsilon)}+O\left( \sqrt{\frac{\delta}{\rho_m(\epsilon)}}+\delta\right),
\end{aligned}
\end{equation}
where $\mathcal{L}_{m}$, $\rho_{m}(\epsilon)$, $C_{*}$ are defined in \eqref{Emn}, \eqref{gamma_m} and \eqref{def_C0}, respectively. 
\end{theorem}

\begin{remark}
Consequently, since $k_{b}=\frac{\omega}{v_{b}}$, it follows from \eqref{equ:lambda_m} that the corresponding wavenumber satisfying
$$k_{b1}\sim \sqrt{\delta}, \quad\mbox{and}~~~ k_{b2}\sim \sqrt{\delta\rho_{m}(\epsilon)}.$$
 \end{remark}

\begin{remark}If we choose $\epsilon$ as an appropiate function of $\delta$, then we can obtain the asymptotic behavior of $\omega_n$. For example, if we choose $\epsilon\sim e^{-\delta^{\beta-1}}$  for $m=2$;  while $\epsilon\sim \delta^{(1-\beta)/(1-\frac{2}{m})}$ for $m>2$, where $0<\beta<1$, then we have $\omega_1\sim\sqrt\delta$, and $\omega_2\sim\sqrt{\delta^{\beta}}$.
 \end{remark}

\begin{proof}[Proof of Theorem \ref{the:main}]
Substituting $\overline{C}_{12}=\sigma_1-\overline{C}_{11}$ and $\overline{C}_{21}=\sigma_2-\overline{C}_{22}$
into \eqref{equ:eig} yields
\begin{align}\label{equ:sigma_2}
\lambda_n&=\frac{1}{2}\left(\overline{C}_{11}+\overline{C}_{22}+(-1)^n\sqrt{(\overline{C}_{11}+\overline{C}_{22})^2-4\sigma_1\overline{C}_{22}-4\sigma_2\overline{C}_{11}+4\sigma_1\sigma_2}\right) \nonumber\\
&=\frac{\overline{C}_{11}+\overline{C}_{22}}{2}\left(1+(-1)^n\sqrt{1-\frac{4\overline{C}_{11}\sigma_2+4\overline{C}_{22}\sigma_1-4\sigma_1\sigma_2}{(\overline{C}_{11}+\overline{C}_{22})^2}}\right).
\end{align}
Due to Proposition \ref{the:1}, and \eqref{sigma_1}, we know that
$$\left|\frac{4\overline{C}_{11}\sigma_2+4\overline{C}_{22}\sigma_1-4\sigma_1\sigma_2}{(\overline{C}_{11}+\overline{C}_{22})^2}\right|\ll 1,\quad\mbox{if}~\epsilon\ll 1.$$
Then, by using the Taylor expansion, $\sqrt{1-x}=1-\frac{1}{2}x+\frac{1}{8}x^{2}+O(x^{3})$, at $x=0$,  
\begin{equation*}\label{lambda1}
\lambda_1=\frac{\overline{C}_{11}\sigma_2+\overline{C}_{22}\sigma_1}{\overline{C}_{11}+\overline{C}_{22}}+O(\frac{1}{C_{11}+C_{22}})=C_{*}+O(\frac{1}{C_{11}+C_{22}}), 
\end{equation*}
and
\begin{equation*}\label{lambda2}
\lambda_2=\overline{C}_{11}+\overline{C}_{22}-\frac{\overline{C}_{11}\sigma_2+\overline{C}_{22}\sigma_1}{\overline{C}_{11}+\overline{C}_{22}}+O(\frac{1}{C_{11}+C_{22}})= \overline{C}_{11}+\overline{C}_{22}-C_{*}+O(\frac{1}{C_{11}+C_{22}}).
\end{equation*}
By virtue of Proposition \ref{the:1}, Lemma \ref{lem:26} and \eqref{sigma_1} again, we finish the proof  of Theorem \ref{the:main}.
\end{proof}
 
\begin{remark}\label{remark1}
Because $\rho_m(\epsilon)\gg1$ when $\epsilon\ll1$, we can deduce from \eqref{equ:lambda_m} that the two subwavelength resonant frequencies have different asymptotic behaviors. It shows that $|\omega_i(\delta)|\to 0$ as $\delta\to0$ under the assumption of Theorem \ref{the:main}, but $\frac{|\omega_2|}{|\omega_1|}\sim\rho_m(\epsilon)\gg1$ by \eqref{equ:lambda_m}, which means that $\omega_1$ and $\omega_2$ are of different scales.     
\end{remark}

\begin{remark}\label{rem1.8}  Theorem \ref{the:main} also holds for the ellipsoidal resonators. After a rotation of the coordinates if necessary, if
\begin{equation*}\label{h1h23}
(h_{1}-h_{2})(x')=\Lambda_1x_1^2+\Lambda_2x_2^2+O(|x'|^{2+\alpha}),\quad\,|x'|\leq\,2R_{0},
\end{equation*}
where $\mathrm{diag}(\Lambda_{1},\Lambda_2)=\nabla_{x'}^{2}(h_{1}-h_{2})(0')$, then the asymptotics of $\omega_2$ in \eqref{equ:lambda_m} is replaced by $$\omega_2=\sqrt{\delta v_b^2(\frac{1}{|D_1|}+\frac{1}{|D_2|})\frac{\pi}{\sqrt{\Lambda_1\Lambda_2}}|\log\epsilon|}+O\left( \sqrt{\frac{\delta}{|\log\epsilon|}}+\delta\right).$$ 
Thus, we can find the effect from the Gaussian curvature. For more details, see \cite{LLY}.
\end{remark}

The rest of this paper is organized as follows. We first give some preliminary results for the Dirichlet problem of the Laplace equation in an exterior domain and then extend the gradient estimates result in \cite{L,LLY} to exterior domains in Section \ref{sec2}. In Section \ref{sec3}, we mainly prove Proposition \ref{the:1}. Finally, we present the corresponding gradient blow-up results in Section \ref{sec4}.

\section{Preliminaries}\label{sec2}

This section is devoted to proving some basic results for the problem in the exterior domains $\mathbb{R}^{3}\backslash\overline{D_1\cup D_2}$ or $\R^3\backslash\overline{D_1^0\cup D_2^0}$. For simplicity, we assume that
$$(h_1-h_2)(x')=\Lambda|x'|^m,\quad m\geq2, \quad\mbox{for}~~ |x'|<2R_0.$$

\begin{lemma}\label{prop:maxv1}
Suppose $v_i$ is a solution of \eqref{equ:vj}. Then $0< v_i<1$ in $\mathbb{R}^{3}\backslash\overline{D_1\cup D_2}$. 
\end{lemma}

\begin{proof}
Fix $x_0\in \mathbb{R}^{3}\backslash\overline{D_1\cup D_2}$. By virtue of the fact that $v_i\sim O(\frac{1}{|x|})$ as $|x|\to \infty$, then for any  $0<\epsilon'<1$, we can choose sufficiently large $R_{\epsilon'}>|x_0|$, such that $|v_i(x)|_{x\in\partial B_R}<\epsilon'$. Because
$$\Delta v_{i}=0,\quad\mbox{in}~~ \mathbb{R}^{3}\backslash\overline{D_1\cup D_2},$$ then, by using the maximum principle in the region $B_R\backslash(D_1\cup D_2)$, we have $-\epsilon'< v_i(x_0)\leq1$. Because $\epsilon'$ and $x_0$ are arbitrary, the proof is finished.
\end{proof}

\begin{lemma}\label{green}
Suppose $v_i$ is a solution of \eqref{equ:vj}. Then for sufficiently large $r_{0}$, such that $D_1\cup D_2\subset B_{r_{0}}(0)$,  we have 
\begin{equation}\label{equ:u_i}
v_i(x)=\int_{\partial B_{r_{0}}(0)}P(x,y)v_i(y)\mathop{}\!\mathrm{d}S_y,\quad\mbox{for}~ x\in \R^3\backslash B_{r_{0}}(0),
\end{equation}
where the Poison kernel $P(x,y)=\frac{\partial G}{\partial \nu}$, $G(x,y)=K(x-y)-K(\frac{|x|}{r_{0}}(x^*-y))$, $K(x)$ is the fundamental solution of the Laplace equation in dimension three, and $x^*=(\frac{r_{0}^2}{|x|^2})x$. Moreover, the following derivative estimates hold when $|x|\to\infty$,
\begin{equation}\label{gradient_infty}
|\nabla v_i|\sim O(\frac{1}{|x|^2}),\quad |\nabla^2 v_i|\sim O(\frac{1}{|x|^3}),\quad\mbox{as}~  |x|\to\infty.
\end{equation}
\end{lemma}

\begin{proof}
For the solution $v_{i}$ of \eqref{equ:vj}, we define $$u_i(x)=\int_{\partial B_{r_{0}}}P(x,y)v_i(y)\mathrm{d}S_y,\quad\mbox{for}~ x\in \R^3\backslash B_{r_{0}}(0).$$
Clearly, 
$$\Delta u_i=0,\quad\mbox{in}~~ \R^3\backslash \overline{B_{r_{0}}(0)}.$$
By the same way as in the proof of theorem 2.6 in \cite{GT}, for $x_0\in\partial B_r$, we have
$$ u_i(x)\to v_i(x_0)\quad\mbox{as}~ x\to x_0, ~x\in \R^3\backslash \overline{B_{r_{0}}(0)}.$$
For $x\in \R^3\backslash \overline{B_{r_{0}}(0)}$, $P(x,y)=\frac{|x|^2-r_{0}^2}{\omega_3r_{0}}\frac{1}{|x-y|^3}$, in view of Lemma \ref{prop:maxv1}, $0\leq v_i\leq1$, we have 
$$u_i(x)\sim\frac{1}{|x|},\quad\mbox{ as}~  x\to \infty.$$
Thus, the differences $w_{i}:=u_i-v_i$ verify
\begin{equation*}
\begin{cases}
\Delta w_{i}=0,\quad&\mbox{in}~  \R^3 \backslash B_{r_{0}}(0),\\
 w_{i}=0,\quad&\mbox{on}~ \partial B_{r_{0}}(0),\\
 w_{i}(x)=O(|x|^{-1}),\quad &\mbox{as} ~ |x|\rightarrow \infty,
\end{cases} 
\end{equation*}
which implies $v_i(x)=u_i(x)$ in $\R^3 \backslash B_{r_{0}}(0)$. Then by a direct computation, we get \eqref{gradient_infty}.
\end{proof}

For $0\leq\,r\leq\,2R_{0}$, let
$$ \Omega_r:=\left\{(x',x_{3})\in \mathbb{R}^{3}~\big|~h_{2}(x')<x_{3}<\varepsilon+h_{1}(x'),~|x'|<r\right\}.$$
To prove Propostion \ref{the:1}, we introduce the Keller-type function $\bar{v}_{1}\in{C}^{2,\alpha}(\mathbb{R}^{n})$, {\color{red}as} in \cite{L,bll} for instance, such that $\bar{v}_{1}=1$ on $\partial{D}_{1}$, $\bar{v}_{1}=0$ on $\partial{D}_{2}$, and $\bar{v}_1(x)=O(\frac{1}{|x|})$ as $|x|\to \infty$, especially,
\begin{align}\label{ubar}
&\bar{v}_{1}(x)
=\frac{x_{3}-h_{2}(x')}{d(x')},\quad\mbox{in}~~~\Omega_{2R_{0}},
\end{align}
where $d(x'):=\varepsilon+(h_{1}-h_{2})(x')$, and satisfies
\begin{align}\label{vbar2}
0<\bar{v}_{1}<1,\quad\mbox{in}~\Omega,\quad \int_{\Omega\backslash\Omega_{R_0/2}}|\Delta \bar{v}_1|<C, \quad\mbox{and}~\|\bar{v}_{1}\|_{C^{2,\alpha}(\Omega\setminus \Omega_{R_{0}})}\leq\,C.
\end{align}

In view of \eqref{ubar}, a direct calculation gives
\begin{equation}\label{nablau_bar}
\begin{split}
\left|\partial_{x_{j}}\bar{v}_{1}(x)\right|&\leq\frac{C|x'|^{m-1}}{\varepsilon+(h_{1}-h_{2})(x')},~~j=1,2,\\
\quad
\partial_{x_{3}}\bar{v}_{1}(x)&=\frac{1}{\varepsilon+(h_{1}-h_{2})(x')},
\end{split}\qquad\qquad~x\in\Omega_{2R_{0}}.
\end{equation}

Although $v_1$ is defined in an unbounded domain in this paper, the same result as Proposition 1.7 in \cite{L} holds.  
\begin{lemma}\label{prop_v1}
Suppose $v_1$ is a solution of \eqref{equ:vj}. Then
\begin{equation}\label{nabla_w_in}
\|\nabla(v_{1}-\bar{v}_{1})\|_{L^{\infty}(\Omega )}\leq\,C.
\end{equation}
\end{lemma}

\begin{proof}
Since $v_{1}-\bar{v}_{1}$ satisfies
\begin{equation*}
\left\{
\begin{aligned}
\Delta (v_{1}-\bar{v}_{1})&=-\Delta \bar{v}_{1},\quad\quad\mbox{in}~   \mathbb{R}^{3}\backslash\overline{D_1\cup D_2},\\
 v_{1}-\bar{v}_{1}&=0,\,\quad\quad\quad\quad\mbox{on}~  \partial D_1\cup\partial D_2,\\
 v_{1}-\bar{v}_{1}&=O(|x|^{-1}),\quad\mbox{as}~   |x|\to \infty.\\
\end{aligned} 
\right.
\end{equation*}
By using Lemma \ref{prop:maxv1}, Lemma \ref{green}, and \eqref{vbar2}, applying the standard elliptic regularity theory, we have 
$$\|v_{1}-\bar{v}_{1}\|_{C^{2,\alpha}(\Omega\setminus \Omega_{R_{0}/2})}\leq\,C.$$
Then employing the same argument as the proof of proposition 1.7 in \cite{L} yields \eqref{nabla_w_in}.
\end{proof}

An immediate consequence of Lemma \ref{prop_v1} is as follows.

\begin{corollary} Under the assumptions of Lemma \ref{prop_v1}, we have
\begin{equation}\label{nabla_v1_in}
\frac{1}{Cd(x')}\leq|\nabla{v}_{1}(x',x_{3})|\leq\frac{C}{d(x')},\quad\qquad~(x',x_{3})\in\Omega_{R_{0}},
\end{equation}
and
\begin{equation}\label{nabla_v1_o}
\|\nabla v_{1}\|_{L^{\infty}(\Omega\setminus\Omega_{R_{0}})}\leq\,C,
\end{equation}
where $d(x')$ is defined in \eqref{ubar}.
\end{corollary}

In order to investigate the asymptotic expansions of $C_{ii}$, we need to consider the following limit problem
\begin{equation}\label{equ:v10}
\begin{cases}
\Delta v_{1}^{0}=0, &\mbox{in}~  \R^3\backslash\overline{D_1^0\cup D_2^0},\\
 v_{1}^{0}=1, &\mbox{on}~  \partial D_1^0\backslash\{0\},\\
 v_{1}^{0}=0, &\mbox{on}~  \partial D_2^0,\\
 v_{1}^{0}(x)=O(\frac{1}{|x|}),   & |x|\rightarrow \infty.
\end{cases} 
\end{equation}

Similarly as lemma 3.1 in \cite{bly}, we have

\begin{lemma}\label{lem:lim}
There exists a unique solution $v_{1}^{0}\in L^{\infty}(\R^3\backslash\overline{D_1^0\cup D_2^0})\cap C^1(\R^3\backslash(D_1^0\cup D_2^0\cup\{0\}))\cap C^{2}(\R^3\backslash\overline{D_1^0\cup D_2^0})$ to \eqref{equ:v10}.
\end{lemma}

\begin{proof}
It is easy to find a subharmonic function $0\leq\,v\leq1$ in $\R^3\backslash\overline{D_1^0\cup D_2^0}$, such that $v(x)=O(\frac{1}{|x|})$ at infinity and satisfying the boundary condition $v\big|_{\partial(D_1^0\cup D_2^0)}\leq\,v_{1}^{0}\big|_{\partial(D_1^0\cup D_2^0)}$. Then, the existence of solution $v_{1}^{0}$ can be easily obtained by the Perron method, see theorem 2.12 and lemma 2.13 in \cite{GT}. The regularity for $v_{1}^{0}$ follows from 
the standard elliptic theory, due to the fact that the boundary of the domain is $C^{2}$ except the origin, see e.g. Theorem 10.2 in \cite{adn0}.
\end{proof}

For $0\leq\,r\leq\,2R_{0}$, let
$$ \Omega_r^0:=\left\{(x',x_{3})\in \mathbb{R}^{3}\backslash\{0\}~\big|~h_{2}(x')<x_{3}<h_{1}(x'),~|x'|<r\right\}.$$
Similarly as $\bar{v}_{1}$, we construct $\bar{v}_{1}^{0}$, such that  $\bar{v}_{1}^{0}(x)=O(\frac{1}{|x|})$, as $|x|\rightarrow \infty$, $\bar{v}_{1}^{0}=1$ on $\partial{D}_{1}^{0}\setminus\{0\}$, $\bar{v}_{1}^{0}=0$ on $\partial{D}_{2}^{0} $,
\begin{equation}\label{def_barv10}
\bar{v}_{1}^{0}=\frac{x_{3}-h_{2}(x')}{(h_{1}-h_{2})(x')},\quad\mbox{in}~~\Omega^{0}_{R_{0}}, \end{equation}
and $\|\bar{v}_{1}^{0}\|_{C^{2,\alpha}(\Omega^{0}\setminus\Omega^{0}_{R_{0}})}\leq\,C$. It is easy to see that
\begin{equation}\label{nablau_starbar}
\left|\partial_{x'}\bar{v}_{1}^{0}(x)\right|\leq\frac{C}{|x'|},\quad
\partial_{x_{3}}\bar{v}_{1}^{0}(x)=\frac{1}{(h_{1}-h_{2})(x')},\quad\mbox{in}~~\Omega_{R_{0}}^{0}.
\end{equation}
It follows from the Proof of Lemma \ref{prop_v1} that 
\begin{equation*}
\left\|\nabla(v_{1}^{0}-\bar{v}_{1}^{0})\right\|_{L^{\infty}(\Omega^{0}_{R_{0}})}\leq\,C.
\end{equation*}
This implies that $\nabla\bar{v}_{1}^{0}$ is also the main term of $\nabla{v}_{1}^{0}$.

\begin{lemma}\label{lem_29}
Let $v_{1}$ and $v_{1}^{0}$ be defined by \eqref{equ:vj} and \eqref{equ:v10}, respectively. Then
\begin{equation}\label{v1-v1*}
\|v_{1}-v_{1}^{0}\|_{L^{\infty}\Big(\R^3\setminus{\big(D_{1}\cup{D}_{2}\cup{D}_{1}^{0}\cup\Omega_{\varepsilon^{1/(2m)}}\big)}\Big)}\leq\,C\varepsilon^{1/2},\qquad\,i=1,2.
\end{equation}
\end{lemma}
\begin{proof}
Since it is obvious that $v_{1}=v_{1}^{0}=0$ on $\partial{D}_{2}=\partial{D}_{2}^{0}$, and $v_1-v_1^0=O(\frac{1}{|x|})$ as $|x|\to\infty$, we only need to estimate the values of $v_1-v_1^0$ on $\partial(D_{1}\cup{D}_{1}^{0})$ and at infinity. We divide $\partial(D_{1}\cup{D}_{1}^{0})$ into two parts: (a) $\partial{D}_{1}^{0}\setminus{D}_{1}$ and (b) $\partial{D}_{1}\setminus{D}_{1}^{0}$. 

For $x\in\partial{D}_{1}^{0}\cap\{\varepsilon^{1/(2m)}<|x'|<R,|x_{3}|<R^{2m}\}$, recalling $v_1^0=1$ on $\partial{D}_{1}^{0}$, and $v_1=1$ on $\partial{D}_{1}$, by virtue of mean value theorem and \eqref{nabla_v1_in},
\begin{equation*}
|v_{1}(x)-v_{1}^{0}(x)|=|v_{1}(x)-1|=|v_{1}(x',h_{1}(x'))-v_{1}(x',\varepsilon+h_{1}(x'))|\leq\frac{C\varepsilon}{\varepsilon+|x'|^{m}}\leq\,C\varepsilon^{1/2}.
\end{equation*}
For $x\in\partial{D}_{1}^{0}\cap(\Omega\setminus\Omega_{R_{0}})$, by using again the mean value theorem and the estimates $\|\nabla v_1\|_{L^{\infty}(\Omega\setminus\Omega_{R_{0}})}\leq\,C$, we have
\begin{equation}\label{v11}
|v_{1}(x)-v_{1}^{0}(x)|\leq\,C\varepsilon.
\end{equation} Similarly, for $x\in\partial{D}_{1}\setminus{D}_{1}^{0}$, the above estimate \eqref{v11} also holds. Hence, 
$$|v_{1}(x)-v_{1}^{0}(x)|\leq\,C\varepsilon^{1/2},\quad\mbox{on}~\partial(D_{1}\cup{D}_{1}^{0}).
$$
On the other hand, for $x\in\{|x'|=\varepsilon^{1/(2m)}\}\cap\Omega_{R_{0}}^{0}$, using the mean value theorem again, there holds
$$|v_{1}(x)-v_{1}^{0}(x)|\leq|v_{1}(x)-\bar{v}_{1}(x)|+|\bar{v}_{1}(x)-\bar{v}_{1}^{0}(x)|+|\bar{v}_{1}^{0}-v_{1}^{0}(x)|\leq\,C\varepsilon^{1/2}.$$

Finally, note that $v_1-v_1^0=O(\frac{1}{|x|})$ as $|x|\to\infty$.  
Hence, by the same argument as Lemma \ref{prop:maxv1}, we obtain that 
$$\max_{\R^3\setminus{(D_{1}\cup{D}_{2}\cup{D}_{1}^{0}\cup\Omega_{\varepsilon^{1/(2m)}})}}(v_1-v_1^0)=\max_{(\{|x'|=\varepsilon^{1/(2m)}\}\cap\Omega_{R_{0}}^{0})\cup\partial(D_{1}\cup{D}_{1}^{0})}(v_1-v_1^0).$$ 
So, \eqref{v1-v1*} holds.
\end{proof}

If $\partial{D}_{1}$ and $\partial{D}_{2}$ are assumed to be $C^{2,\alpha}$ and satisfy \eqref{equ:D1D2}, then we have an improvement of Lemma \ref{lem_29} by interpolation. The proof is omitted here, because it is the same as the proof of lemma 3.2 in \cite{L} with the help of Lemma \ref{prop:maxv1} and Lemma \ref{green}.

\begin{lemma}\label{lem_49}
Assume that $v_{1}$ and $v_{1}^{0}$ are solutions of \eqref{equ:vj} and \eqref{equ:v10}, respectively. If $\partial{D}_{1}^{0}$ and $\partial{D}_{2}^{0}$ are $C^{2,\alpha}$ and satisfy \eqref{equ:D1D2}, then 
\begin{equation}\label{nabla_v1v10}
\begin{split}
&|\nabla{v}_{1}(x)|\leq\,C|x'|^{-m},~\,x\in\Omega_{R_0}\setminus\Omega_{\varepsilon^{1/(2m)}},\\
&|\nabla{v}_{1}^{0}(x)|\leq\,C|x'|^{-m},~\,x\in\Omega^{0}_{R_0}\setminus\Omega^{0}_{\varepsilon^{1/(2m)}};
\end{split}
\end{equation}
and
\begin{equation}\label{nabla_v1-v10}
|\nabla(v_{1}-v_{1}^{0})(x)|\leq\,C\varepsilon^{1/4}|x'|^{-m},\quad\mbox{in}~~~\Omega^{0}_{R_0}\setminus\Omega^{0}_{\varepsilon^{1/(2m)}}.
\end{equation}
\end{lemma}

\section{Proof of Proposition \ref{the:1}}\label{sec3}

This section is devoted to proving Proposition \ref{the:1}. We prove part (i) and part (ii) in Subsection \ref{subsec4.1} and Subsection \ref{subsec4.2}, respectively. Here the asymptotic behavior at infinity is mainly dealt with.

\subsection{Proof of (i) in Proposition \ref{the:1}}\label{subsec4.1}

Let $$\bar{d}:=\max_{x\in\overline{D_1^0 \cup D_2^0}}|x|.$$
Then
\begin{lemma}\label{lem:difi}
For sufficiently small $\epsilon<<\bar{d}$, we have $D_1\cup D_2\subset B_{2\bar{d}}(0)$, moreover,
\begin{equation}
\int_{\R^3\backslash B_{2\bar{d}}(0)}|\nabla v_1|^2 =\int_{\R^3\backslash B_{2\bar{d}}(0)}|\nabla v_1^0|^2+O(\sqrt\epsilon).
\end{equation}
\end{lemma}

\begin{proof}
First,
$$\int_{\R^3\backslash B_{2\bar{d}}(0)}|\nabla v_1|^2 -\int_{\R^3\backslash B_{2\bar{d}}(0)}|\nabla v_1^0|^2=\int_{\R^3\backslash B_{2\bar{d}}(0)}\nabla(v_1- v_1^0)\cdot\nabla(v_1+ v_1^0).$$
By virtue of proposition 2.74 and 2.75 in \cite{GB}, we have 
$$\frac{\partial v_1}{\partial r}\sim O(\frac{1}{|x|^2}), \quad\frac{\partial v_1^0}{\partial r}\sim O(\frac{1}{|x|^2}),\quad\mbox{as}~ |x|\to\infty.$$ 
Applying the divergence theorem, we have
$$\left|\int_{\R^3\backslash B_{2\bar{d}}(0)}\nabla(v_1- v_1^0)\cdot\nabla(v_1+ v_1^0)\right|=\left|\int_{\partial B_{2\bar{d}}}(v_1-v_1^0)\cdot\frac{\partial (v_1+v_1^0)}{\partial \nu}\mathop{}\!\mathrm{d}S_x\right|\leq\,C\varepsilon^{1/2},$$
 by virtue of \eqref{v1-v1*} and the estimates $\|\nabla v_1\|_{L^{\infty}(\Omega\setminus\Omega_{R_{0}})}+\|\nabla v_1^0\|_{L^{\infty}(\Omega^{0}\setminus\Omega_{R_{0}}^{0})}\leq\,C$.
\end{proof}

\begin{proof}[Proof of (i) in Proposition \ref{the:1} ]
To prove (i) in Proposition \ref{the:1}, we first divide the integral into two parts:
\begin{equation}\label{energy_v1}
C_{ii}=\int_{\R^3\backslash \overline{D}}|\nabla{v}_{1}|^{2}
=\int_{\Omega\backslash\Omega_{\varepsilon^{\gamma}}}|\nabla{v}_{1}|^{2}+\int_{\Omega_{\varepsilon^{\gamma}}}|\nabla{v}_{1}|^{2}
=:\mathrm{I}+\mathrm{II},
\end{equation}
where $\Omega=\R^3\backslash \overline{D}$, $D=D_{1}\cup{D}_{2}$, and we take $\gamma=\frac{1}{4m}$, for convenience.

{\bf STEP 1.} First, we use the explicit functions $v_{1}^{0}$ and $\bar{v}_{1}$ to approximate ${v}_{1}$ in the unbounded region $\Omega\backslash\Omega_{\varepsilon^{\gamma}}$ and the small region $\Omega_{\varepsilon^{\gamma}}$, respectively. We claim that
\begin{equation}\label{equ346}
\begin{aligned}
\mathrm{I}=\int_{ \Omega^0\backslash\Omega_{\varepsilon^{\gamma}}^0}|\nabla{v}_{1}^0|^{2}+O(E_{m}(\epsilon)),\quad\mbox{and}~~\mathrm{II}=\int_{\Omega_{\varepsilon^{\gamma}}}|\nabla\bar{v}_{1}|^{2}+O(E_{m}(\epsilon)).
\end{aligned}
\end{equation}
where $E_{m}(\epsilon)$ is defined by \eqref{Emn}.

Indeed, we further divide $\Omega\backslash\Omega_{\varepsilon^{\gamma}}$ into two parts:
\begin{equation}
\begin{aligned}
&\mathrm{I}-\int_{ \Omega^0\backslash\Omega_{\varepsilon^{\gamma}}^0}|\nabla{v}_{1}^0|^{2}\\
=&\Big(\int_{\Omega\setminus\Omega_{R_{0}}}|\nabla{v}_{1}|^{2}-\int_{ \Omega^0\backslash\Omega_{R_0}^0}|\nabla{v}_{1}^0|^{2}\Big)+\Big(\int_{\Omega_{R_{0}}\setminus\Omega_{\varepsilon^{\gamma}}}|\nabla{v}_{1}|^{2}-\int_{\Omega_{R_{0}}^{0}\setminus\Omega_{\varepsilon^{\gamma}}^0}|\nabla{v}_{1}^0|^{2}\Big)\\
=&:\mathrm{I}_1+\mathrm{I}_2.
\end{aligned}
\end{equation}
For $\mathrm{I}_{1}$, we rearrange it as follows:
\begin{align*}
\mathrm{I}_1&=\int_{\Omega\setminus{({D}_{1}^{0}\cup\Omega_{R_{0}})}}(|\nabla{v}_{1}|^{2}-|\nabla{v}_{1}^{0}|^{2})+\int_{D_{1}^{0}\setminus(D_{1}\cup\Omega_{R_{0}})}|\nabla{v}_{1}|^{2}+\int_{D_{1}\setminus{D_{1}^{0}}}|\nabla{v}_{1}^{0}|^{2}\\
&=:\mathrm{I}_{11}+\mathrm{I}_{12}+\mathrm{I}_{13}.
\end{align*}
By virtue of Lemma \ref{lem:difi}, we have
\begin{align*}
|\mathrm{I}_{11}|&=\left|\int_{B_{2d}\setminus{(D\cup{D}_{1}^{0}\cup\Omega_{R_{0}})}}(|\nabla{v}_{1}|^{2}-|\nabla{v}_{1}^{0}|^{2})\right|+O(\sqrt{\epsilon})\\
&\leq\left|\int_{B_{2d}\setminus{(D\cup{D}_{1}^{0}\cup\Omega_{R_{0}})}}\nabla(v_{1}-v_{1}^{0})\cdot\nabla(v_{1}+v_{1}^{0})\right|+O(\sqrt{\epsilon}).
\end{align*}
Since
$$\Delta(v_{1}-v_{1}^{0})=0,\quad\mbox{in}~~~\Omega\setminus{\big({D}_{1}^{0}\cup\Omega_{R_{0}/2}\big)},$$
and
$$0<v_{1},v_{1}^{0}<1,\quad\mbox{in}~~~\Omega\setminus{\big({D}_{1}^{0}\cup\Omega_{R_{0}/2}\big)},$$
it follows by standard elliptic regularity theory that  
$$|\nabla^{2}(v_{1}-v_{1}^{0})|\leq|\nabla^{2}v_{1}|+|\nabla^{2}v_{1}^{0}|\leq\,C,\quad\mbox{in}~~~\Omega\setminus{\big({D}_{1}^{0}\cup\Omega_{R_{0}/2}\big)},$$
where $C$ is independent of $\varepsilon$, provided $\partial{D}_{1}^{0}$, $\partial{D}_{2}^{0}$ 
are $C^{2,\alpha}$.

By interpolation inequalities, see e.g. Lemma 6.32 and Lemma 6.35 in \cite{GT}, for $x\in \Omega\setminus{({D}_{1}^{0}\cup\Omega_{R_{0}/2})}$, $0<\mu<1$, we have 
\begin{equation*}
 |\nabla(v_{1}-v_{1}^{0})(x)|\leq \frac{C}{\mu}\left\|v_{1}-v_{1}^{0}\right\|_{L^{\infty}(\Omega\setminus{({D}_{1}^{0}\cup\Omega_{R_{0}/2}))}}+\mu\left\|\nabla^{2}(v_{1}-v_{1}^{0})\right\|_{L^{\infty}(\Omega\setminus{({D}_{1}^{0}\cup\Omega_{R_{0}/2}))}}, 
\end{equation*}
Taking $\mu=\epsilon^{\frac{1}{4}}$, by using the interpolation above with \eqref{v1-v1*}, we have
\begin{equation*}\label{v1-v1*outside}
|\nabla(v_{1}-v_{1}^{0})|\leq\,C
\varepsilon^{1/4},\quad\mbox{in}~~~\Omega\setminus{\big({D}_{1}^{0}\cup\Omega_{R_{0}/2}\big)}.
\end{equation*}
So, $\mathrm{I}_{11}=O(\epsilon^{\frac{1}{4}})$. In view of the boundedness of $|\nabla{v}_{1}|$ in $D_{1}^{0}\setminus(D_{1}\cup\Omega_{R_{0}})$ and  $D_{1}\setminus{D_{1}^{0}}$, and their volumes $|D_{1}^{0}\setminus(D_{1}\cup\Omega_{R_{0}})|$ and $|D_{1}\setminus{D_{1}^{0}}|$ are less than $C\varepsilon$, it is easy to get $\mathrm{I}_{12}+\mathrm{I}_{13}\leq\,C\varepsilon$. Hence, $|\mathrm{I}_{1}|\leq\,C\varepsilon^{1/4}$.
 
For $\mathrm{I}_2$, 
\begin{equation}\label{equ:I2}
\mathrm{I}_2=\int_{(\Omega_{R_{0}}\setminus\Omega_{\varepsilon^{\gamma}})\setminus(\Omega^{0}_{R_0}\setminus\Omega^{0}_{\varepsilon^{\gamma}})}|\nabla{v}_{1}|^{2}+\int_{\Omega^{0}_{R_0}\setminus\Omega^{0}_{\varepsilon^{\gamma}}}|\nabla(v_{1}-v^{0}_{1})|^{2}\nonumber+2\int_{\Omega^{0}_{R_0}\setminus\Omega^{0}_{\varepsilon^{\gamma}}}\nabla{v}_{1}^{0}\cdot\nabla(v_{1}-v^{0}_{1}), 
\end{equation}
by using Lemma \ref{lem_49}, we obtain
$$\int_{(\Omega_{R_{0}}\setminus\Omega_{\varepsilon^{\gamma}})\setminus(\Omega^{0}_{R_0}\setminus\Omega^{0}_{\varepsilon^{\gamma}})}|\nabla{v}_{1}|^{2}\leq\,\int_{\varepsilon^{\gamma}<|x'|<R_{0}}\frac{C\varepsilon\, dx'}{|x'|^{2m}}\leq\,C\varepsilon^{\frac{1}{2}+\frac{1}{2m}}\leq\,CE_{m}(\varepsilon);$$
\begin{align}\nonumber\int_{\Omega^{0}_{R_0}\setminus\Omega^{0}_{\varepsilon^{\gamma}}}|\nabla(v_{1}-v^{0}_{1})|^{2}\leq& C\epsilon^{1/2}\int_{{\Omega_{R_0}^0}\setminus\Omega^{0}_{\varepsilon^{\gamma}}}|x'|^{-2m}dx'dx_n \\
\leq& C\epsilon^{1/2}\int_{\epsilon^\gamma<|x'|<R_0}\frac{dx'}{|x'|^m}\leq C\epsilon^{1/4}E_m(\epsilon);\nonumber
\end{align}
and
$$\left|2\int_{\Omega^{0}_{R_0}\setminus\Omega^{0}_{\varepsilon^{\gamma}}}\nabla{v}_{1}^{0}\cdot\nabla(v_{1}-v^{0}_{1})\right|\leq C\epsilon^{1/4}\int_{\epsilon^\gamma<|x'|<R_0}\frac{dx'}{|x'|^m}\leq CE_m(\epsilon)$$
so
$$|\mathrm{I}_2|\leq\,CE_{m}(\varepsilon).$$
Thus, the claim for $\mathrm{I}$ is proved.

For $\mathrm{II}$, by  Lemma \ref{prop_v1}, we have
\begin{equation}\label{v1barv1}
\mathrm{II}-\int_{\Omega_{\varepsilon^{\gamma}}}|\nabla\bar{v}_{1}|^{2}=2\int_{\Omega_{\varepsilon^{\gamma}}}\nabla\bar{v}_{1}\cdot\nabla(v_{1}-\bar{v}_{1})
+\int_{\Omega_{\varepsilon^{\gamma}}}|\nabla(v_{1}-\bar{v}_{1})|^{2}
=O(\epsilon^{\frac{1}{2m}}).
\end{equation}
So the proof of claim \eqref{equ346} is finished.

{\bf STEP 2.}
Next we prove that 
$$\int_{ \Omega^0\backslash\Omega_{\varepsilon^{\gamma}}^0}|\nabla{v}_{1}^0|^{2}=\int_{\Omega^{0}_{R_{0}}\setminus\Omega^{0}_{\varepsilon^{\gamma}}}|\partial_{x_{3}}\bar{v}_{1}^{0}|^{2}+M_{1}^{(1)}+M_{1}^{(2)}+O(\varepsilon^{\gamma}),$$
where 
\begin{align*}
M_{1}^{(1)}&:=\int_{\Omega^0\setminus\Omega^{0}_{R_0}}|\nabla{v}_{1}^{0}|^{2}, \\
M_{1}^{(2)}&:=2\int_{\Omega^{0}_{R_{0}}}\nabla\bar{v}_{1}^{0}\cdot\nabla(v^{0}_{1}-\bar{v}_{1}^{0})
+\int_{\Omega^{0}_{R_{0}}}(|\nabla(v^{0}_{1}-\bar{v}_{1}^{0})|^{2}+|\partial_{x'}\bar{v}_1^0|^2).
\end{align*}
Indeed, similarly as Lemma \ref{gradient_infty}, we have $D_iv_1^0(x)\sim O(\frac{1}{|x|^2})$, and $v_1^0(x)\sim O(\frac{1}{|x|})$. Combining with $|\nabla v_1^0|<C$ in $\Omega^0\setminus\Omega^{0}_{R_0}$, $|v_1^0|\leq1$ and the divergence theorem, we obtain $M_{1}^{(1)}<C$. For $M_{1}^{(2)}$, using $\|\nabla(v^{0}_{1}-\bar{v}_{1}^{0})\|_{L^{\infty}(\Omega_{R})}\leq\,C$, $$|\nabla\bar{v}_{1}^{0}(x)|\leq\frac{C}{(h_{1}-h_{2})(x')},\quad\,\mbox{and}~~|\partial_{x'}\bar{v}_{1}^{0}(x)|^{2}\leq\frac{C}{(h_{1}-h_{2})(x')},$$ we have $M_{1}^{(2)}<C$, and
$$2\left|\int_{\Omega^{0}_{\varepsilon^{\gamma}}}\nabla\bar{v}_{1}^{0}\cdot\nabla(v^{0}_{1}-\bar{v}_{1}^{0})\right|
+\int_{\Omega^{0}_{\varepsilon^{\gamma}}}(|\nabla(v^{0}_{1}-\bar{v}_{1}^{0})|^{2}+|\partial_{x'}\bar{v}_1^0|^2)\leq\,C\varepsilon^{\gamma}.$$
On the other hand, recalling \eqref{nablau_bar}, we have
$$\int_{\Omega_{\varepsilon^{\gamma}}}|\nabla\bar{v}_{1}|^{2}=\int_{\Omega_{\varepsilon^{\gamma}}}|\partial_{x_3}\bar{v}_{1}|^{2}+O(E_{m}(\epsilon)),$$
here we used
$$\int_{\Omega_{\varepsilon^{\gamma}}}|\partial_{x'}\bar{v}_{1}|^{2} \le C \int_{|x'|<\varepsilon^\gamma} \frac{|x'|^{2m-2}}{\varepsilon + |x'|^{m}} \,dx'  \le C \int_{|x'|<\varepsilon^\gamma} \,|x'|^{m-2}dx' = O\left(\varepsilon^{\frac{n+m-3}{4m}}\right).$$

While, it follows from \eqref{nablau_starbar} and \eqref{nablau_bar} that 
\begin{align}\label{equ3.17}
&\int_{\Omega_{\varepsilon^{\gamma}}}|\partial_{x_{3}}\bar{v}_{1}|^{2}+\int_{\Omega^{0}_{R_{0}}\setminus\Omega^{0}_{\varepsilon^{\gamma}}}|\partial_{x_{3}}\bar{v}_{1}^{0}|^{2}\nonumber\\=&\int_{R_0>|x'|>\varepsilon^{\gamma}}\frac{dx'}{(h_{1}-h_{2})(x')}+\int_{|x'|<\varepsilon^{\gamma}}\frac{dx'}{\varepsilon+(h_{1}-h_{2})(x')}\nonumber\\
=&\int_{\varepsilon^\gamma < |x'| < R_0} \frac{dx'}{\Lambda|x'|^{m}} +  \int_{|x'| < \varepsilon^\gamma} \frac{dx'}{\varepsilon + \Lambda|x'|^{m}}\nonumber\\
= &\int_{|x'| <R_0} \frac{dx'}{\varepsilon + \Lambda|x'|^{m}}+ O\left(\varepsilon^{\frac{1}{2}+\frac{1}{2m}}  \right).
\end{align}
Thus, to prove \eqref{equ:m>2}, it suffices to calculate the integral in the right hand side of \eqref{equ3.17}.

{\bf STEP 3.} We now calculate the integral in \eqref{equ3.17}. 

$(i)$ For $m=2$, 
\begin{align}\label{n=3part2}
\begin{split}
\int_{|x'| <R_0} \frac{dx'}{\varepsilon + \Lambda|x'|^{2}} =&\frac{2\pi}{\Lambda}\int_{0}^{R_{0}(\frac{\Lambda}{\varepsilon})^{1/2}}\frac{r\,dr}{1+r^{2}}=\frac{\mathcal{L}_{2}|\log\varepsilon|}{\Lambda} +M_{1}^{(3)}+ O(\varepsilon),
\end{split}
\end{align} 
where $M_{1}^{(3)}:=\frac{\pi}{\Lambda}(2\log R_{0}+\log\Lambda)$. Therefore, it follows from \eqref{energy_v1}, \eqref{equ346}, and \eqref{n=3part2} that \eqref{equ:m>2} holds,
where $M_{1}:=M_{1}^{(1)}+M_{1}^{(2)}+M_{1}^{(3)}$.

$(ii)$ For $m>2$,
\begin{align*}\label{n=3part2}
\int_{|x'| <R_0} \frac{dx'}{\varepsilon + \Lambda|x'|^{m}} =&\frac{2\pi}{\Lambda^{\frac{2}{m}}\varepsilon^{1-\frac{2}{m}}}\int_{0}^{R_{0}(\frac{\Lambda}{\varepsilon})^{1/m}}\frac{r\,dr}{1+r^{m}}=\frac{\mathcal{L}_{m}}{\Lambda^{\frac{2}{m}}\varepsilon^{1-\frac{2}{m}}} +M_{1}^{(3)}+ O(\varepsilon^{2-\frac{2}{m}}),
\end{align*} 
where $M_{1}^{(3)}:=\frac{2\pi}{\Lambda}R_{0}^{2-m}$. Therefore, \eqref{equ:m>2} holds with $M_{1}:=M_{1}^{(1)}+M_{1}^{(2)}+M_{1}^{(3)}$. Finally, we show that the constant $M_{1}$ is independent of the choice of $R_{0}$. If not, suppose that there exist $M_{1}(R_{0})$ and $M_{1}(\tilde{R}_{0})$, both independent of $\varepsilon$, such that \eqref{equ:m>2} holds, then
$$|M_{1}(R_{0})-M_{1}(\tilde{R}_{0})|\leq\,CE_{m}(\epsilon)\rightarrow0,\quad\mbox{as}~~\varepsilon\rightarrow0,$$
which implies that $M_{1}(R_{0})=M_{1}(\tilde{R}_{0})$. For $C_{22}$, the proof is the same. Thus, the proof of (i) in Proposition \ref{the:1} is completed.
\end{proof}

\subsection{Proof of (ii) in Proposition \ref{the:1}}\label{subsec4.2}

\begin{lemma}\label{equ:c12}
$C_{12}=C_{21}$, where $C_{ij}$ is defined as \eqref{equ:cij}. 
\end{lemma}

\begin{proof}
For any large $R>>\bar{d}$, By appying the divergence theorem, we have
\begin{align*}
0&=\int_{B_R\backslash D}v_2\Delta v_1-\int_{B_R\backslash D}v_1\Delta v_2 
=\int_{\partial(B_R\backslash D)}v_2\frac{\partial v_1}{\partial \nu}- \int_{\partial(B_R\backslash D)}v_1\frac{\partial v_2}{\partial \nu} \\
&=-\int_{\partial D_2}\frac{\partial v_1}{\partial \nu}+\int_{\partial D_1}\frac{\partial v_2}{\partial \nu}+\int_{\partial{B}_R}v_2\frac{\partial v_1}{\partial \nu}- \int_{\partial{B}_R}v_1\frac{\partial v_2}{\partial \nu} \\
&=C_{12}-C_{21}+\int_{\partial{B}_R}v_2\frac{\partial v_1}{\partial \nu}- \int_{\partial{B}_R}v_1\frac{\partial v_2}{\partial \nu}.
\end{align*}
Using $v_{i}=O(\frac{1}{|x|})$ as $|x|\to\infty$, and Proposition 2.74 and 2.75 in \cite{GB}, which imply that $\frac{\partial v_i}{\partial r}\sim O(\frac{1}{|x|^2})$ as $|x|\to\infty$, and due to the arbitrariness of $R$, we finish the proof.
\end{proof}

\begin{proof}[Proof of (ii) in Proposition \ref{the:1}]
By using Lemma \ref{equ:c12}, and applying the divergence theorem on $B_R$, such that $\overline {D_1\cup D_2}\subset B_R$, we have
\begin{equation}\label{equ5.1}
C_{11}+C_{12}=C_{11}+C_{21}=-\int_{\partial D_1}\frac{\partial v_1}{\partial\nu}\Big|_+-\int_{\partial D_2}\frac{\partial v_1}{\partial\nu}\Big|_+= -\int_{\partial B_R}\frac{\partial v_1}{\partial \nu}\Big|_+.  
\end{equation}
Recalling the Poisson formula, \eqref{equ:u_i},
\begin{equation*}
\begin{aligned}
\frac{\partial v_1}{\partial \nu}\Big|_{x\in\partial B_R}&=\sum_{k=1}^n {\partial}_kv_1\cdot\frac{x_k}{R},\\
&=\int_{\partial B_{r_{0}}}\frac{1}{\omega_3r_{0}|x-y|^3}(-R+\frac{3r_{0}^2}{R}-3\frac{(R^2-r_{0}^2)\sum(x_k-y_k)y_k}{|x-y|^2R})v_1(y)dS_y,
\end{aligned}
\end{equation*}
where $r_{0}$ is fixed in \eqref{equ:u_i}. Now $|y|=r_{0}$, $|x|=R$, since $v_1>0$ in $\R^3\backslash\overline{(D_1\cup D_2)}$,  we choose $R>> r_{0}$, such that 
\begin{equation*}
\frac{-2}{\omega_3r_{0}R^2}\int_{\partial B_{r_{0}}}v_1(y)dS_y<\frac{\partial v_1}{\partial \nu}\Big|_{x\in\partial B_R}<\frac{-1}{2\omega_3r_{0}R^2}\int_{\partial B_{r_{0}}}v_1(y)dS_y.
\end{equation*}
By using \ref{lem:difi},
\begin{equation*}
\frac{-2}{\omega_3r_{0}R^2}\Big(\int_{\partial B_{r_{0}}}v_1^0(y)dS_y+O(\sqrt\epsilon)\Big)<\frac{\partial v_1}{\partial \nu}\Big|_{x\in\partial B_R}<\frac{-1}{2\omega_3r_{0}R^2}\Big(\int_{\partial B_{r_{0}}}v_1^0(y)dS_y+O(\sqrt\epsilon)\Big),
\end{equation*}
Integrating $\frac{\partial v_1}{\partial \nu}$ on $\partial B_R$, by \eqref{equ5.1}, yields 
$$\frac{1}{C}<C_{11}+C_{12}<C.$$
The proof for $C_{21}+C_{22}$ is similar. So the proof of Proposition \ref{the:1} is completed.
\end{proof}

\section{ Gradient Blow-up}\label{sec4}

It is known the eigenmode is approximately a constant on each resonator. If these constants are different, then as the two resonators are moved close together, the gradient of the field between them will blow up. As a direct consequence of Theorem \ref{the:main}, 
we  can also show the gradient blow-up of the eigenmode similarly as in \cite{ADY}.

As in \eqref{equ:eig}, the eigenvector of $\overline{C}$ associated to the eigenvalue $\lambda_n$ is given by
\begin{equation} \label{eq:eigenvectors}
\left(\frac{\lambda_n-\overline{C}_{22}}{\overline{C}_{21}},\,1\,\right).
\end{equation}
From the leading order behaviour of $\lambda_n$, \eqref{equ:sigma_2}, and of the capacitance coefficients \eqref{equ:cij} we have that, for $\epsilon\ll1$,
\begin{equation} \label{eq:d1value}
    \frac{\lambda_n-\overline{C}_{22}}{\overline{C}_{21}}=\begin{cases}
    1+O(\frac{1}{C_{ii}}), & n=1,\\
    -\frac{|D_2|}{|D_1|}+O(\frac{1}{C_{ii}}), & n=2.
    \end{cases}
\end{equation}
Then the eigenmodes are given by
\begin{equation} \label{eq:mode_defn}
    u_n(x)=\S_D[\phi_n](x)+O(\omega_n),\quad\,n=1,2.
\end{equation}
By using \eqref{eq:psi_basis} in Lemma \ref{lem:26}, we have
\begin{equation} \label{eq:mode_densities}
    \phi_n:=\frac{\lambda_n-\overline{C}_{22}}{\overline{C}_{21}}\psi_1+\psi_2.
\end{equation}

Note that $v_j=S_D[\psi_j]$. Then 
\begin{equation}
\begin{aligned}
\S_D[\phi_n](x)&=\frac{\lambda_n-\overline{C}_{22}}{\overline{C}_{21}}\S_D[\psi_1](x)+\S_D[\psi_2](x) \\
&=\frac{\lambda_n-\overline{C}_{22}}{\overline{C}_{21}}v_1(x)+v_2(x).
\end{aligned}
\end{equation}
Hence,
\begin{equation} \label{eq:dvalue}
    u_n(x)=\begin{cases}
    \frac{\lambda_n-\overline{C}_{22}}{\overline{C}_{21}}+O(\omega_n), & x\in \D_1, \\
    1+O(\omega_n), & x\in \D_2.
    \end{cases}
\end{equation}
Therefore, for sufficiently small $\delta>0$, $u_1|_{\D_1}$ and $u_1|_{\D_2}$ have the same sign whereas $u_2|_{\D_1}$ and $u_2|_{\D_2}$ have different signs.

Further,
\begin{equation}
\nabla \S_D[\phi_n](x)=\left(\frac{\lambda_n-\overline{C}_{22}-\overline{C}_{21}}{\overline{C}_{21}}\right)\nabla  v_1(x)+\nabla (v_1+v_2)(x).
\end{equation}
Since $(v_1+v_2)|_{\D_1}=(v_1+v_2)|_{\D_2}=1$, and $0<v_1+v_2<1$, due to Lemma~2.3 of \cite{bly} (or the result in \cite{ammari2007estimates,llby}), we have 
$$\|\nabla (v_1+v_2)\|_{L^{\infty}(\mathbb{R}^{3}\backslash\overline{D_1\cup D_2})}\leq\,C.$$ 
Thus, due to Proposition \ref{the:1}, and \eqref{sigma_1} in Proposition \ref{the:1}, we know that the main singular part of $\nabla \S_D[\phi_n](x)$ is $\left(\frac{\lambda_n-\overline{C}_{22}}{\overline{C}_{21}}-1\right) \nabla v_1(x)$.

Combining with \eqref{eq:d1value}, as a consequence of Proposition \ref{the:1}, we obtain the following gradient estimates for the two eigenmodes $u_{1}$ and $u_{2}$ similarly as in \cite{ADY}. Namely, choose $\epsilon=\epsilon(\delta)$ so that $\epsilon\to0$ as $\delta\to0$, if the two leading order values are different then the maximum of the gradient of the presure between the two resonators must blow up as $\delta\to0$.

\begin{theorem} \label{thm:gradient_bounds}
Let $u_1$ and $u_2$ denote the sub-wavelength eigenmodes for two m-convex $C^{2,\alpha}$ resonators $D_{1}$ and $D_{2}$ in $\mathbb{R}^{3}$, 
which are $\epsilon$-apart as in Theorem \ref{the:main}. Normalize $u_1$ and $u_2$ such that
    \begin{equation*}
        \lim_{\delta\to0} \left\|u_1\right\|_{L^2(\partial D)}\sim1,\qquad \lim_{\delta\to0} \left\|u_2\right\|_{L^2(\partial D)}\sim1,\quad\mbox{for}~x\in\partial{D}_{1}\cup\partial{D}_{2}.
    \end{equation*}
Then suppose that the distance $\epsilon$ satisfies $\epsilon\sim e^{-1/\delta^{1-\beta}}$ for $m=2$, $\epsilon\sim \delta^{(1-\beta)/(1-\frac{2}{m})}$ for $m>2$, for $0<\beta<1$, we have, as $\delta\to0$, 
\begin{equation*}
\begin{aligned}
 &\max_{x\in\mathbb{R}^{3}\setminus\overline{D_{1}\cup{D}_{2}}}|\nabla u_1(x)|\prec
 \begin{cases}
 \frac{1}{\epsilon|\log\epsilon|},&\mbox{if}~m=2,\\
\epsilon^{-\frac{2}{m}},&\mbox{if}~m>2;
\end{cases}
\end{aligned}
\end{equation*}
 and
\begin{equation*}
     \max_{x\in\mathbb{R}^{3}\setminus\overline{D_{1}\cup{D}_{2}}}|\nabla u_2(x)|\sim \frac{1}{\epsilon}. 
\end{equation*}
\end{theorem}

\noindent{\bf{\large Acknowledgements.}} We thank the referees for carefully reading and many helpful suggestions which improve the exposition.

\end{document}